\documentclass[11pt]{article}
\usepackage{amsfonts,latexsym,rawfonts,amsmath,amssymb,amsthm,graphicx}
\numberwithin{equation}{section}
\newtheorem{theorem}{Theorem}[section]
\newtheorem{lem}[theorem]{Lemma}
\newtheorem{thm}[theorem]{Theorem}

\newtheorem{cor}[theorem]{Corollary}

\newtheorem{rem}[theorem]{Remark}

\title{Inequalities for eigenvalues of Laplacian and biharmonic operators on submanifolds}
\author{Yong Luo${}$\footnote{Corresponding author. Email address: yongluo-math@cqut.edu.cn},  Xianjing Zheng}
\date{}
\begin{document}
	\maketitle
	\begin{abstract}In this paper we study eigenvalues of Laplacian and biharmonic operators on compact domains in complete manifolds. We establish several new inequalities for eigenvalues of Laplacian and
biharmonic operators respectively  by using Sobolev type inequalities.

	\end{abstract}
	\section{Introduction}
     Let $(M,g)$  be an $n$-dimensional complete Riemannian manifold and $\Omega$  a bounded connected domain in $M$. Let  $\Delta$  and $\Delta^{2}$ denote the Laplacian and biharmonic operators acting on functions on $ M $ respectively. The Dirichlet eigenvalue problem is stated as:
     	\begin{eqnarray}\label{Dirichlet problem}
     	\left\{\begin{array}{c}
     		\Delta u=-\lambda u \text { in } \Omega, \\
     		\left.u\right|_{\partial \Omega}=0 .
     	\end{array}\right.
     \end{eqnarray}
     \indent
     Let
     \begin{eqnarray*}
     	0<\lambda_{1}<\lambda_{2} \leq \lambda_{3} \leq \cdots,
     \end{eqnarray*}
     denote the successive eigenvalues of (\ref{Dirichlet problem}). Here each eigenvalue is repeated according to its multiplicity.\\
     \indent
     Generally, if an equality of eigenvalus requires no hypotheses on the geometric quantities of the domain $\Omega$ (other than its dimension), it is called a universal inequality. Universal inequality is an important area in the study of the spectrum of $\Delta$ and many works have been done during the past decades. Now we give a brief introduction of the main results.\\
     \indent
     The study of universal inequalities for (\ref{Dirichlet problem}) was initiated by Payne, P\'olya, and Weinberger \cite{Payne} in 1956, they proved
     \begin{eqnarray*}
     	\lambda_{k+1}-\lambda_{k} \leq \frac{2}{k} \sum_{i=1}^{k} \lambda_{i}
     \end{eqnarray*}
     for $\Omega \subset \mathbb{R}^{2} $. Exactly by the same proof, their result can promoted to dimension $n$ as
     \begin{eqnarray*}
     	\lambda_{k+1}-\lambda_{k} \leq \frac{4}{k n} \sum_{i=1}^{k} \lambda_{i}
     \end{eqnarray*}
     for $ \Omega \subset \mathbb{R}^{n}$.\\
     \indent
     In 1980, Hile and Protter \cite{Hile} proved
     \begin{eqnarray*}
      \sum_{i=1}^{k} \frac{\lambda_{i}}{\lambda_{k+1}-\lambda_{i}} \geq \frac{k n}{4}.
     \end{eqnarray*}\\
     \indent
      In 1991, Yang \cite{Yang} proved
     \begin{eqnarray*}
     	\sum_{i=1}^{k}\left(\lambda_{k+1}-\lambda_{i}\right)\left(\lambda_{k+1}-\left(1+\frac{4}{n}\right) \lambda_{i}\right) \leq 0,
     \end{eqnarray*}
     which is called Yang's first inequality. It can be derived from this inequality that
     \begin{eqnarray*}
     \lambda_{k+1} \leq \frac{1}{k}\left(1+\frac{4}{n}\right) \sum_{i=1}^{k} \lambda_{i},
     \end{eqnarray*}
     which is called Yang's second inequality.\\
     \indent
   When  $M$ is isometrically immersed in $R^m$ , Chen and Cheng \cite{Chen} and Soufi, Harrel and Ilias \cite{Sou}  have proved independently that
   \begin{eqnarray}
   	\sum_{i=1}^{k}\left(\lambda_{k+1}-\lambda_{i}\right)^{2} \leq \frac{4}{n} \sum_{i=1}^{k}\left(\lambda_{k+1}-\lambda_{i}\right)\left(\lambda_{i}+\frac{n^{2}}{4} H_{0}^{2}\right),
   \end{eqnarray}
       where $ H_{0}=\max _{x \in \Omega}|\mathbf{H}|$ and $ \mathbf{H} $ is the mean curvature vector of $M$ .\\
       \indent
       Next we consider the eigenvalue problem for the biharmonic operator or the so-called clamped plate problem, which is given by
       \begin{eqnarray}	\label{clamped plate problem}
       	\left\{\begin{array}{l}
       		\Delta^{2} u=\Gamma u \text { in } \Omega, \\
       		\left.u\right|_{\partial \Omega}=\left.\frac{\partial u}{\partial \nu}\right|_{\partial \Omega}=0 .
       	\end{array}\right.
       \end{eqnarray}
       where $\nu$ denotes the outward unit normal to the boundary $\partial \Omega$.\\
       \indent
       The study on universal inequalities for eigenvalues of the clamped plate problem can be traced  back to 1950s. When  M  is a Euclidean space $ \mathbf{R}^{n}$ , Payne, P\'olya and Weinberger \cite{Payne} in 1956 proved that the eigenvalues $\left\{\Gamma_{i}\right\}_{i=1}^{\infty}$ of the problem (\ref{clamped plate problem}) satisfy
       \begin{eqnarray}\label{Payne's function}
       \Gamma_{k+1}-\Gamma_{k} \leq \frac{8(n+2)}{n^{2}} \frac{1}{k} \sum_{i=1}^{k} \Gamma_{i} .
       \end{eqnarray}\\
       \indent
       In 1984, Hile and Yeh \cite{Yeh} strengthened (\ref{Payne's function}) and obtained
       \begin{eqnarray*}
       	 \frac{n^{2} k^{3 / 2}}{8 n+2}\left(\sum_{i=1}^{k} \Gamma_{i}\right)^{-\frac{1}{2}} \leq \sum_{i=1}^{k} \frac{\Gamma_{i}^{\frac{1}{2}}}{\Gamma_{k+1}-\Gamma_{i}}.
       \end{eqnarray*}\\
       \indent
       In 1990, Chen and Qian \cite{Qian},  and Hook \cite{Hook} independently proved the following inequality
      \begin{eqnarray*}
      \frac{n^{2} k^{2}}{8(n+2)} \leq\left(\sum_{i=1}^{k} \Gamma_{i}^{1 / 2}\right)\left(\sum_{i=1}^{k} \frac{\Gamma_{i}^{1 / 2}}{\Gamma_{k+1}-\Gamma_{i}}\right) .
      \end{eqnarray*} \\
      \indent
      In  Ashbaugh's survey paper \cite{Ashbaugh} on recent developments on eigenvalue problems, he asked if we can establish an analog of Yang's first inequality for the eigenvalues of the clamped plate problem. In 2006, Cheng and Yang solved this problem in \cite{Cheng2} by proving that
      \begin{eqnarray*}
      	\Gamma_{k+1}-\frac{1}{k} \sum_{i=1}^{k} \Gamma_{i} \leq\left[\frac{8(n+2)}{n^{2}}\right]^{\frac{1}{2}} \frac{1}{k} \sum_{i=1}^{k}\left[\Gamma_{i}\left(\Gamma_{k+1}-\Gamma_{i}\right)\right]^{\frac{1}{2}} .
      \end{eqnarray*}\\
      \indent
      In 2010, Cheng, Ichikawa and Mametsuka \cite{Cheng} proved that if  $M$  is an  $n$-dimensional submanifold isometrically immersed in a Euclidean space with mean curvature vector  $\mathbf{H}$ , then
      \begin{eqnarray}\label{Cheng's function}
      	  \sum_{i=1}^{k}\left(\Gamma_{k+1}-\Gamma_{i}\right)^{2} \leq \frac{1}{n^{2}} \sum_{i=1}^{k}\left(\Gamma_{k+1}-\Gamma_{i}\right)\left(n^{2} H_{0}^{2}+(2 n+4) \Gamma_{i}^{1 / 2}\right)\left(n^{2} H_{0}^{2}+4 \Gamma_{i}^{1 / 2}\right),
      \end{eqnarray}
       where  $H_{0}=\sup _{\Omega}|\mathbf{H}|$.\\
       \indent
        In 2011, Wang and Xia \cite{Wang} strengthened the inequality (\ref{Cheng's function}) by Cheng, Ichikawa and Mametsuka and obtained
        \begin{align}
        	\sum_{i=1}^{k}\left(\Gamma_{k+1}-\Gamma_{i}\right)^{2} \leq & \frac{1}{n}\left\{\sum_{i=1}^{k}\left(\Gamma_{k+1}-\Gamma_{i}\right)^{2}\left(n^{2} H_{0}^{2}+(2 n+4) \Gamma_{i}^{1 / 2}\right)\right\}^{1 / 2} \nonumber\\
        	& \times\left\{\sum_{i=1}^{k}\left(\Gamma_{k+1}-\Gamma_{i}\right)\left(n^{2} H_{0}^{2}+4 \Gamma_{i}^{1 / 2}\right)\right\}^{1 / 2}.
        \end{align}
        \\
        \indent
        In this paper we study eigenvalues of the Dirichlet eigenvalue problem and the clamped plate problem on compact domains in complete Riemannian manifolds.
        When the manifolds are isometrically immersed in a Euclidean space,
        by using Sobolev type inequalities, we obtain several new  inequalities which involve  integral of the mean curvature vector field. We have
	\begin{thm}\label{Theorem 1.1}
		Let $(M,g)$ be an $n$-dimensional complete Riemannian manifold $(n\ge 3)$ and let $\Omega $ be a bounded domain with smooth boundary $\partial \Omega$ in $M$. Let $ \mathcal{D}(M) $ denotes the space of $ C^{\infty} $ functions with compact support in $M$. Assume that $M$ is isometrically immersed in $R^m$ with mean curvature vector $\mathbf{H}$, and there exists a constant $C_{1}>0$  such that for any  $u \in \mathcal{D}(M)$ ,
		\begin{eqnarray}
			\left(\int_{M}|u|^{2 n /(n-2)} \right)^{(n-2) / n} \leq C_{1} \int_{M}|\nabla u|^{2}.
		\end{eqnarray}
		Denote by $\lambda_{i}$ the $ith$ eigenvalue of the problem (\ref{Dirichlet problem}).
	Then
	\begin{eqnarray}\label{first result}
		n \sum_{i=1}^{k}\left(\lambda_{k+1}-\lambda_{i}\right)^{2} \leq \left(4+n^{2}C_1\||\mathbf{H}|\|^2_{L^n(\Omega)}\right) \sum_{i=1}^{k}\left(\lambda_{k+1}-\lambda_{i}\right)\lambda_{i},
	\end{eqnarray}
where $\||\mathbf{H}|\|_{L^n(\Omega)}=(\int_\Omega|H|^ndvol_g)^\frac{1}{n}.$
	\end{thm}
	\begin{thm}\label{Theorem 1.2}
        Under the same assumption as in the Theorem 1.1, and denote by $\Gamma_{i}$ the $ith$ eigenvalue of the clamped plate problem (\ref{clamped plate problem}).
	Then
	\begin{align}\label{second result}
		\sum_{i=1}^{k}\left(\Gamma_{k+1}-\Gamma_{i}\right)^{2} \leq & \frac{1}{n}\left\{\sum_{i=1}^{k}\left(\Gamma_{k+1}-\Gamma_{i}\right)^{2}\left(n^{2}C_{1}\||\mathbf{H}|\|^2_{L^n(\Omega)}+2n+4) \Gamma_{i}^{1 / 2}\right)\right\}^{1 / 2} \nonumber\\
		& \times\left\{\sum_{i=1}^{k}\left(\Gamma_{k+1}-\Gamma_{i}\right)\left(n^{2}C_{1}\||\mathbf{H}|\|^2_{L^n(\Omega)}+4)
		\Gamma_{i}^{1 / 2}\right)\right\}^{1 / 2}.
	\end{align}
	\end{thm}
Since $\||\mathbf{H}|\|_{L^n(\Omega)}\leq\||\mathbf{H}|\|_{L^n(M)}$, directly from Theorems 1.1, 1.2 we obtain the following universal inequalities.
\begin{cor}
		Let $(M,g)$ be an $n$-dimensional complete Riemannian manifold $(n\ge 3)$ and let $\Omega $ be a bounded domain with smooth boundary $\partial \Omega$ in $M$. Let $ \mathcal{D}(M) $ denotes the space of $ C^{\infty} $ functions with compact support in $M$. Assume that $M$ is isometrically immersed in $R^m$ with mean curvature vector $\mathbf{H}$, and there exists a constant $C_{1}>0$  such that for any  $u \in \mathcal{D}(M)$ ,
		\begin{eqnarray*}
			\left(\int_{M}|u|^{2 n /(n-2)} \right)^{(n-2) / n} \leq C_{1} \int_{M}|\nabla u|^{2}.
		\end{eqnarray*}
Denote by $\lambda_{i}$ the $ith$ eigenvalue of the problem (\ref{Dirichlet problem}).
	Then
	\begin{eqnarray*}	
		n \sum_{i=1}^{k}\left(\lambda_{k+1}-\lambda_{i}\right)^{2} \leq \left(4+n^{2}C_1\||\mathbf{H}|\|^2_{L^n(M)}\right) \sum_{i=1}^{k}\left(\lambda_{k+1}-\lambda_{i}\right)\lambda_{i}.
	\end{eqnarray*}
\end{cor}
\begin{cor}
       Under the same assumption as in the Theorem 1.1, and denote by $\Gamma_{i}$ the $ith$ eigenvalue of the clamped plate problem (\ref{clamped plate problem}).
	Then
	\begin{align*}
		\sum_{i=1}^{k}\left(\Gamma_{k+1}-\Gamma_{i}\right)^{2} \leq & \frac{1}{n}\left\{\sum_{i=1}^{k}\left(\Gamma_{k+1}-\Gamma_{i}\right)^{2}\left(n^{2}C_{1}\||\mathbf{H}|\|^2_{L^n(M)}+2n+4) \Gamma_{i}^{1 / 2}\right)\right\}^{1 / 2} \nonumber\\
		& \times\left\{\sum_{i=1}^{k}\left(\Gamma_{k+1}-\Gamma_{i}\right)\left(n^{2}C_{1}\||\mathbf{H}|\|^2_{L^n(M)}+4)
		\Gamma_{i}^{1 / 2}\right)\right\}^{1 / 2}.
	\end{align*}
\end{cor}
	\begin{rem}
		Let  $(M, g)$  be a complete Riemannian $n$-manifold of infinite volume. Let $ \mathcal{D}(M) $ be the space of $ C^{\infty} $ functions with compact support in $ M$ and $\lambda_{1}^{D}(M) $ denotes the first eigenvalue of the Laplacian for the Dirichlet problem on $\Omega$.The following two statements are equivalent.\\
		(1) There exists  $C_{0}>0$  such that for any  $u \in \mathcal{D}(M)$ ,
		\begin{eqnarray*}
			\left(\int_{M}|u|^{2 n /(n-2)} \right)^{(n-2) / n} \leq C_{0} \int_{M}|\nabla u|^{2}
		\end{eqnarray*}
		(2) There exists  $\Lambda>0$ such that for any $ \Omega \subset \subset M $, $\lambda_{1}^{D}(\Omega) \geq \Lambda V o l_{g}(\Omega)^{-2 / n}$ .
	\end{rem}
	\begin{thm}\label{Theoerm 1.4}
	Let $(M,g)$ be an  $n$-dimensional complete Riemannian manifold and  $\Omega $  a bounded domain with smooth boundary $\partial \Omega$ in $M$. Assume that $M$ is isometrically immersed in $R^m$ with mean curvature vector $\mathbf{H}$. Suppose that $ \||\mathbf{H}|\|_{L^n(\Omega)} C_{2}<1$ , where $ C_{2} $ is the constant in Michael-Simon's Sobolev inequality. Denote by $\lambda_{i}$ the $ith$ eigenvalue of the problem (\ref{Dirichlet problem}).
	Then
		\begin{eqnarray}	\label{third result}
		n \sum_{i=1}^{k}\left(\lambda_{k+1}-\lambda_{i}\right)^{2} \leq \left(4+n^{2}C_s\left\||\mathbf{H}|\right\|_{L^n(\Omega)}^2\right) \sum_{i=1}^{k}\left(\lambda_{k+1}-\lambda_{i}\right)\lambda_{i},
	\end{eqnarray}
	where  $C_{s}=\left(\frac{C_{2}}{1-\||\mathbf{H}|\|_{L^n(\Omega)} C_{2}} \frac{2(n-1)}{n-2}\right)^{2} $.
	\end{thm}
	\begin{thm}\label{Theorem 1.5}
	  Under the same assumption as in the Theorem 1.6, and denote by $\Gamma_{i}$ the $ith$ eigenvalue of the clamped plate problem (\ref{clamped plate problem}).
	  Then
	  	\begin{align}\label{fourth result}
	  		\sum_{i=1}^{k}\left(\Gamma_{k+1}-\Gamma_{i}\right)^{2} \leq & \frac{1}{n}\left\{\sum_{i=1}^{k}\left(\Gamma_{k+1}-\Gamma_{i}\right)^{2}\left(n^{2}C_{s}\left\||\mathbf{H}|\right\|_{L^n(\Omega)}+2n+4) \Gamma_{i}^{1 / 2}\right)\right\}^{1 / 2} \nonumber\\
	  		& \times\left\{\sum_{i=1}^{k}\left(\Gamma_{k+1}-\Gamma_{i}\right)\left(n^{2}C_{s}\left\||\mathbf{H}|\right\|_{L^n(\Omega)}+4)
	  		\Gamma_{i}^{1 / 2}\right)\right\}^{1 / 2},
	  	\end{align}
	  	where  $C_{s}=\left(\frac{C_{2}}{1-\||\mathbf{H}|\|_{L^n(\Omega)} C_{2}} \frac{2(n-1)}{n-2}\right)^{2} $.
	\end{thm}
	From Theorems 1.6, 1.7, we obtain the following universal inequalities.
\begin{cor}
Let $(M,g)$ be an  $n$-dimensional complete Riemannian manifold and  $\Omega $  a bounded domain with smooth boundary $\partial \Omega$ in $M$. Assume that $M$ is isometrically immersed in $R^m$ with mean curvature vector $\mathbf{H}$. Suppose that $ \||\mathbf{H}|\|_{L^n(M)} C_{2}<1$ , where $ C_{2} $ is the constant in Michael-Simon's Sobolev inequality. Denote by $\lambda_{i}$ the $ith$ eigenvalue of the problem (\ref{Dirichlet problem}).
	Then
		\begin{eqnarray*}	
		n \sum_{i=1}^{k}\left(\lambda_{k+1}-\lambda_{i}\right)^{2} \leq \left(4+n^{2}C_s\left\||\mathbf{H}|\right\|_{L^n(M)}^2\right) \sum_{i=1}^{k}\left(\lambda_{k+1}-\lambda_{i}\right)\lambda_{i},
	\end{eqnarray*}
	where  $C_{s}=\left(\frac{C_{2}}{1-\||\mathbf{H}|\|_{L^n(M)} C_{2}} \frac{2(n-1)}{n-2}\right)^{2} $.
\end{cor}

	\begin{cor}
	  Under the same assumption as in the Theorem 1.6, and denote by $\Gamma_{i}$ the $ith$ eigenvalue of the clamped plate problem (\ref{clamped plate problem}).
	  Then
	  	\begin{align*}
	  		\sum_{i=1}^{k}\left(\Gamma_{k+1}-\Gamma_{i}\right)^{2} \leq & \frac{1}{n}\left\{\sum_{i=1}^{k}\left(\Gamma_{k+1}-\Gamma_{i}\right)^{2}\left(n^{2}C_{s}\left\||\mathbf{H}|\right\|^2_{L^n(M)}+2n+4) \Gamma_{i}^{1 / 2}\right)\right\}^{1 / 2} \nonumber\\
	  		& \times\left\{\sum_{i=1}^{k}\left(\Gamma_{k+1}-\Gamma_{i}\right)\left(n^{2}C_{s}\left\||\mathbf{H}|\right\|^2_{L^n(M)}+4)
	  		\Gamma_{i}^{1 / 2}\right)\right\}^{1 / 2},
	  	\end{align*}
	  	where  $C_{s}=\left(\frac{C_{2}}{1-\||\mathbf{H}|\|_{L^n(M)} C_{2}} \frac{2(n-1)}{n-2}\right)^{2} $.
	\end{cor}

\section{Michael-Simon's type  Sobolev inequality}
	The Michael-Simon's Sobolev inequality is well known as the following lemma:
    \begin{lem}[\cite{JH}]
    Let $M$  be an $n$-dimensional submanifold immersed in  $\mathbb{R}^{n+p}$ with mean carvature vector $\mathbf{H}$. Then for any function $ h \in   C_{0}^{1}(M)$ , we have
    \begin{eqnarray*}
    	\left(\int_{M}|h|^{\frac{n}{n-1}} \right)^{\frac{n-1}{n}} \leq C_{2}\left(\int_{M}|\nabla h|+\int_{M}|\mathbf{H}||h| \right),
    \end{eqnarray*}
    where  $C_{2}$  is a constant depends only on dimension $ n$ .
    \end{lem}
    In the latter sections we will need the following corollary proved by Li, Xu and Zhou for $\Omega=M$:
    \begin{lem}[\cite{YW}]
    	Let  $M$  be an $n$-dimensional submanifold immersed in $ \mathbb{R}^{n+p}$ with mean carvature vector $\mathbf{H}$.  Suppose that $\Omega$ is a smooth open domain in $M$ and $ \|\mathbf{H}\|_{L^n(\Omega)} C_{2}<1$ , where $ C_{2} $ is the constant in Lemma $2.1$. Then for any  $f \in C_{0}^{1}(\Omega)$ , we have
    	\begin{eqnarray*}
    			\left(\int_{\Omega}|f|^{\frac{2 n}{n-2}} \right)^{\frac{n-2}{n}} \leq C_{s} \int_{\Omega}|\nabla f|^{2},
    	\end{eqnarray*}
    	where  $C_{s}=\left(\frac{C_{2}}{1-\|\mathbf{H}\|_{L^n(\Omega)}C_{2}} \frac{2(n-1)}{n-2}\right)^{2}$.
    \end{lem}

 \begin{proof}   For the reader's convenience, we give a proof of this lemma here. For a function  $h\in C^1_0(\Omega)$  as in Lemma $2.1$, by H\"older's inequality, we have
    	\begin{eqnarray*}
    		\int_{M}|\mathbf{H}||h|\leq\left(\int_{\Omega}|h|^{\frac{n}{n-1}}\right)^{\frac{n-1}{n}}\|\mathbf{H}\|_{L^n(\Omega)}.
    	\end{eqnarray*}
    Therefore by the Michael-Simon's Sobolev inequality in Lemma $2.1$ and the inequality above we get
    	\begin{eqnarray*}
    		\left(\int_{\Omega}|h|^{\frac{n}{n-1}} \right)^{\frac{n-1}{n}} \leq \frac{C_{2}}{1-\|\mathbf{H}\|_{L^n(\Omega)} C_{2}} \int_{\Omega}|\nabla h| .
    	\end{eqnarray*}
    	Now for any  $f \in C_{0}^{1}(\Omega)$ , let $ h=f^{\frac{2(n-1)}{n-2}}$ , we get
    	\begin{align}
    		\left(\int_{\Omega} f^{\frac{2 n}{n-2}}\right)^{\frac{n-1}{n}} & \leq \frac{C_{2}}{1-\|\mathbf{H}\|_{L^n(\Omega)} C_{2}} \frac{2(n-1)}{n-2} \int_{\Omega}|f|^{\frac{n}{n-2}}|\nabla f| \nonumber\\
    		& \leq \frac{C_{2}}{1-\|\mathbf{H}\| C_{2}} \frac{2(n-1)}{n-2}\left(\int_{M} f^{\frac{2 n}{n-2}} \right)^{\frac{1}{2}}\left(\int_{M}|\nabla f|^{2}\right)^{\frac{1}{2}}\nonumber. \end{align}
    	Then the conclusion follows easily.  Substituting (\ref{Simon's inequality 2}) into (\ref{Ash's function 1}), we can prove (\ref{third result}). \end{proof}
	\section{Proof of Theorem \ref{Theorem 1.1}}
		\begin{proof}
		Assume that $u_i,i=1,...k,$ are eigenfunctions of the Dirichlet problem with eigenvalues $\lambda_i$, such that $$\int_\Omega u_iu_j=\delta_{i j}.$$ The first part of the proof can be found in \cite{Sou}. To prove Theorem 1.1 we will need the following lemma.
	\begin{lem}\label{Stability inequality}
	For any smooth function $F$ on $M$ and any positive integer $k$ one has
	\begin{eqnarray*}
	\sum_{i=1}^{k}\left(\lambda_{k+1}-\lambda_{i}\right)^{2}\left\langle[-\Delta, F] u_{i}, F u_{i}\right\rangle_{L^{2}(\Omega)} \leq \sum_{i=1}^{k}\left(\lambda_{k+1}-\lambda_{i}\right)\left\|[-\Delta, F]
 u_{i}\right\|_{L^{2}(\Omega)}^{2},
	\end{eqnarray*}
where$[\cdot, \cdot]$ denotes the commutator of two operators $B$ and $C$ via $[B,C]=BC-CB $.
	\end{lem}

	This lemma can be found in \cite{Ash}, so we skip the proof here.\\
	\indent
	Now let $X$ $:$ $M$ $\longrightarrow$ $R^m$ be an isometric immersion, and $X_{1}$, $\ldots$ , $X_{m}$ the components of the immersion $X$. A straightforward calculation gives
	\begin{eqnarray*}
	\left[-\Delta, X_{\alpha}\right] u_{i}=\left(-\Delta X_{\alpha}\right) u_{i}-2 \nabla X_{\alpha} \cdot \nabla u_{i} .
	\end{eqnarray*}
	By integrating by parts we obtain
	\begin{eqnarray*}
		\left\langle\left[-\Delta, X_{\alpha}\right] u_{i}, X_{\alpha} u_{i}\right\rangle_{L^{2}(\Omega)}=\int_{\Omega}\left|\nabla X_{\alpha}\right|^{2} u_{i}^{2} .
	\end{eqnarray*}
	thus
	\begin{eqnarray*}
	\sum_{\alpha}\left\langle\left[-\Delta, X_{\alpha}\right] u_{i}, X_{\alpha} u_{i}\right\rangle_{L^{2}(\Omega)}=\sum_{\alpha} \int_{\Omega}\left|\nabla X_{\alpha}\right|^{2} u_{i}^{2}=n \int_{\Omega} u_{i}^{2}=n
	\end{eqnarray*}
	On the other hand, we have
	\begin{eqnarray*}
	\left\|\left[-\Delta, X_{\alpha}\right] u_{i}\right\|_{L^{2}(\Omega)}^{2}=\int_{\Omega}\left(\left(-\Delta X_{\alpha}\right) u_{i}-2 \nabla X_{\alpha} \cdot \nabla u_{i}\right)^{2} .
	\end{eqnarray*}
	Since $X$ is an isometric immersion, it follows that
	\begin{eqnarray*}
	 \left(\Delta X_{1},\ldots, \Delta X_{m}\right)=n\mathbf{H}  \\
	  \sum_{\alpha=1}^{m}\left(\nabla X_{\alpha} \cdot \nabla u_{i}\right)^{2} =\left|\nabla u_{i}\right|^{2} \\
	  \sum_{\alpha=1}^{m}\left(-\Delta X_{\alpha}\right) u_{i} \nabla X_{\alpha} \cdot \nabla u_{i} =\frac n2 \mathbf{H} \cdot \nabla u_{i}^{2}=0.\\
	\end{eqnarray*}
	 Using all these facts, we get
	\begin{eqnarray*}
	\sum_{\alpha=1}^{m}\left\|\left[-\Delta, X_{\alpha}\right] u_{i}\right\|_{L^{2}(\Omega)}^{2}=n^{2}\int_{\Omega}|\mathbf{H}|^{2} u_{i}^{2}+4 \int_{\Omega}\left|\nabla u_{i}\right|^{2}.
	\end{eqnarray*}
	Since
	\begin{eqnarray*}
	\int_{\Omega}\left|\nabla u_{i}\right|^{2}=\lambda_{i},
	\end{eqnarray*}
	using Lemma 3.1 we obtain
	\begin{eqnarray}\label{Ash's function}
	n \sum_{i=1}^{k}\left(\lambda_{k+1}-\lambda_{i}\right)^{2} \leq \sum_{i=1}^{k}\left(\lambda_{k+1}-\lambda_{i}\right)\left(n^{2}\int_{\Omega}|\mathbf{H}|^{2}u_{i}^{2}+4 \lambda_{i}\right),
	\end{eqnarray}
	Using H\"older's inequality, we have
	\begin{eqnarray*}
		\int_{\Omega}|\mathbf{H}|^{2}u_{i}^{2} \leq\left\|u_{i}\right\|_{L^{2 p}(\Omega)}^{2} \left\|\left |\mathbf{H} \right |  \right \|_{L^{\frac{2 p}{p-1}}(\Omega)}^{2}.
	\end{eqnarray*}
    Letting $p= \frac{n}{n-2} $ and using Lemma 2.1 we obtain
	\begin{eqnarray}\label{Simon's inequality 1}
		\int_{\Omega}|\mathbf{H}|^{2}u_{i}^{2} \leq C_{1}\left\|\nabla u_{i}\right\|_{L^{2}(\Omega)}^{2} \left\|\left | \mathbf{H} \right |  \right \|^2_{L^{n}\left(\Omega\right)}.
	\end{eqnarray}
Substituting (\ref{Simon's inequality 1}) into (\ref{Ash's function}), we can prove (\ref{first result}). \end{proof}
   
\section{Proof of Theorem \ref{Theorem 1.2}}
    The first part of the proof can be found in \cite{Wang}. To prove Theorem 1.2, we need the following lemma.
    	\begin{lem}[\cite{Wang}]
    Let $\Gamma_{i}$, i=1, \ldots , k+1, be the $ith$ eigenvalue of the Clamped plate problem  and  $u_{i}$  the orthonormal eigenfunction corresponding to  $\Gamma_{i}$ , that is,
    		\begin{eqnarray*}
    		\Delta^{2} u_{i}=\Gamma_{i} u_{i} \quad \text { in } \Omega,\left.\quad u_{i}\right|_{\partial \Omega}=\left.\frac{\partial u_{i}}{\partial \nu}\right|_{\partial \Omega}=0,
    \\ \int_{M} u_{i} u_{j}=\delta_{i j}, \quad \forall \ i, j=1,2, \ldots,k+1.
    	\end{eqnarray*}
    	Then for any smooth function  $h$$:$ $\Omega$$ \rightarrow$ $\mathbf{R}$ , we have
    		\begin{align}\label{lemma 4.1 function}
    		&\sum_{i=1}^{k}\left(\Gamma_{k+1}-\Gamma_{i}\right)^{2} \int_{\Omega} u_{i}^{2}|\nabla h|^{2}\nonumber \\
    		&\leq \delta \sum_{i=1}^{k}\left(\Gamma_{k+1}-\Gamma_{i}\right)^{2} \int_{\Omega}\left(u_{i}^{2}(\Delta h)^{2}+4\left(\left\langle\nabla h, \nabla u_{i}\right\rangle^{2}+u_{i} \Delta h\left\langle\nabla h, \nabla u_{i}\right\rangle\right)-2 u_{i}|\nabla h|^{2} \Delta u_{i}\right) \nonumber\\
    		\quad&+\sum_{i=1}^{k} \frac{\left(\Gamma_{k+1}-\Gamma_{i}\right)}{\delta} \int_{\Omega}\left(\left\langle\nabla h, \nabla u_{i}\right\rangle+\frac{u_{i} \Delta h}{2}\right)^{2},
    		\end{align}
    	where  $\delta$ is any positive constant.
    	\end{lem}
    	\begin{proof}
     For $i=1, \ldots, k$ , consider the functions  $\phi_{i} : \Omega \rightarrow\mathbf{R}$  given by
    \begin{eqnarray*}
    \phi_{i}=h u_{i}-\sum_{j=1}^{k} r_{i j} u_{j},
    \end{eqnarray*}
    where
    \begin{eqnarray*}
    r_{i j}=\int_{\Omega} h u_{i} u_{j} .
    \end{eqnarray*}
    Since  $\left.\phi_{i}\right|_{\partial \Omega}=\left.\frac{\partial \phi_{i}}{\partial v}\right|_{\partial \Omega}=0 $ and
    \begin{eqnarray*}
    \int_{\Omega} u_{j} \phi_{i}=0, \forall i, j=1, \ldots, k,
    \end{eqnarray*}
    it follows from the Rayleigh-Ritz inequality that
    \begin{eqnarray}\label{Rayleigh Rize's inequality 2}
    \Gamma_{k+1} \leq \frac{\int_{\Omega} \phi_{i} \Delta^{2} \phi_{i}}{\int_{\Omega} \phi_{i}^{2}} .
    \end{eqnarray}
    We have
	\begin{align}\label{fenzi}
			\int_{\Omega}\phi_{i} \Delta^{2} \phi_{i}
			 &=\int_{\Omega} \phi_{i}\left(\Delta^{2}\left(h u_{i}\right)-\sum_{j=1}^{k}  r_{i j} \Gamma_{j} u_{j}\right)\nonumber \\
			 &=\int_{\Omega} \phi_{i} \Delta^{2}\left(h u_{i}\right)\nonumber \\
			 &=\int_{\Omega} \phi_{i}\left(\Delta\left(u_{i} \Delta h\right)+2 \Delta\left\langle\nabla h, \nabla u_{i}\right\rangle+2\left\langle\nabla h, \nabla\left(\Delta u_{i}\right)\right\rangle+\Delta h \Delta u_{i}+\Gamma_{i} h u_{i}\right) \nonumber\\
			 &=\Gamma_{i}\left\|\phi_{i}\right\|^{2}+\int_{\Omega} \phi_{i}\left(\Delta\left(u_{i} \Delta h\right)+2 \Delta\left\langle\nabla h, \nabla u_{i}\right\rangle+2\left\langle\nabla h, \nabla\left(\Delta u_{i}\right)\right\rangle+\Delta h \Delta u_{i}\right)\nonumber\\
			&= \Gamma_{i}\left\|\phi_{i}\right\|^{2}+\int_{\Omega} h u_{i}\left(\Delta\left(u_{i} \Delta h\right)+2 \Delta\left\langle\nabla h, \nabla u_{i}\right\rangle+2\left\langle\nabla h, \nabla\left(\Delta u_{i}\right)\right\rangle+\Delta h \Delta u_{i}\right)\nonumber \\
			\quad&\quad-\sum_{j=1}^{k} r_{i j} s_{i j}
   \end{align}
				where  $\left\|\phi_{i}\right\|^{2}=\int_{\Omega} \phi_{i}^{2}$  and
		\begin{eqnarray*}
			s_{i j}=\int_{\Omega} u_{j}\left(\Delta\left(u_{i} \Delta h\right)+2 \Delta\left\langle\nabla h, \nabla u_{i}\right\rangle+2\left\langle\nabla h, \nabla\left(\Delta u_{i}\right)\right\rangle+\Delta h \Delta u_{i}\right) .
	    \end{eqnarray*}
	    Multiplying the equation $ \Delta^{2} u_{i}=\Gamma_{i} u_{i} $ by $h u_{j}$ , we have
	   	\begin{eqnarray}\label{Gamma i}
	    h u_{j} \Delta^{2} u_{i}=\Gamma_{i} h u_{i} u_{j} .
	    \end{eqnarray}
	    Changing the roles of  $i$  and  $j$ , one gets
	   	\begin{eqnarray}\label{Gamma j}
	    h u_{i} \Delta^{2} u_{j}=\Gamma_{j} h u_{i} u_{j} .
	    \end{eqnarray}
	     Subtracting (\ref{Gamma i}) from (\ref{Gamma j}) and integrating the resulted equation on $ \Omega$ , we get by using the divergence theorem that
	    \begin{align}\label{r s function 2}
	    	\left(\Gamma_{j}-\Gamma_{i}\right) r_{i j} & =\int_{\Omega}\left(h u_{i} \Delta^{2} u_{j}-h u_{j} \Delta^{2} u_{i}\right) \nonumber\\
	    	& =\int_{\Omega}\left(\Delta\left(h u_{i}\right) \Delta u_{j}-\Delta\left(h u_{j}\right) \Delta u_{i}\right)\nonumber \\
	    	& =\int_{\Omega}\left(\left(u_{i} \Delta h+2\left\langle\nabla h, \nabla u_{i}\right\rangle\right) \Delta u_{j}-\left(\left(u_{j} \Delta h+2\left\langle\nabla h, \nabla u_{j}\right\rangle\right) \Delta u_{i}\right)\right.\nonumber \\
	    	& =\int_{\Omega}\left(u_{j}\left(\Delta\left(u_{i} \Delta h\right)+2 \Delta\left\langle\nabla h, \nabla u_{i}\right\rangle\right)-u_{j} \Delta h \Delta u_{i}+2 u_{j} \operatorname{div}\left(\Delta u_{i} \nabla h\right)\right) \nonumber\\
	    	& =\int_{\Omega} u_{j}\left(\Delta\left(u_{i} \Delta h\right)+2 \Delta\left\langle\nabla h, \nabla u_{i}\right\rangle+\Delta h \Delta u_{i}+2\left\langle\nabla \Delta u_{i}, \nabla h\right\rangle\right)\nonumber \\
	    	& =s_{i j},
	    \end{align}
	    where for a vector field  $Z$  on $ \Omega$, $\operatorname{div} Z$  denotes the divergence of $ Z$ . Observe that
	    \begin{align}\label{first part function}
	    	\int_{\Omega} &hu_{i}\left(\Delta\left(u_{i} \Delta h\right)+2 \Delta\left\langle\nabla h, \nabla u_{i}\right\rangle+2\left\langle\nabla h, \nabla\left(\Delta u_{i}\right)\right\rangle+\Delta h \Delta u_{i}\right) \nonumber\\
	    	 &=\int_{\Omega}\left(\Delta\left(h u_{i}\right) u_{i} \Delta h+2 \Delta\left(h u_{i}\right)\left\langle\nabla h, \nabla u_{i}\right\rangle-2 \Delta u_{i} \operatorname{div}\left(h u_{i} \nabla h\right)+h u_{i} \Delta h \Delta u_{i}\right) \nonumber\\
	    	&=\int_{\Omega}\left(u_{i}^{2}(\Delta h)^{2}+4\left(\left\langle\nabla h, \nabla u_{i}\right\rangle^{2}+u_{i} \Delta h\left\langle\nabla h, \nabla u_{i}\right\rangle\right)-2 u_{i}|\nabla h|^{2} \Delta u_{i}\right) .
	    \end{align}
	    It follows from (\ref{Rayleigh Rize's inequality 2}) , (\ref{fenzi}), (\ref{r s function 2}) and (\ref{first part function}) that
	    \begin{align}\label{Rayleigh Rize inequality}
	    	&\left(\Gamma_{k+1}-\Gamma_{i}\right)|| \phi_{i} \|^{2} \nonumber\\
	    	&\quad\leq \int_{\Omega}\left(u_{i}^{2}(\Delta h)^{2}+4\left(\left\langle\nabla h, \nabla u_{i}\right\rangle^{2}+u_{i} \Delta h\left\langle\nabla h, \nabla u_{i}\right\rangle\right)-2 u_{i}|\nabla h|^{2} \Delta u_{i}\right) \nonumber\\
	    	&\quad+\sum_{j=1}^{k}\left(\Gamma_{i}-\Gamma_{j}\right) r_{i j}^{2} .	
	   \end{align}
	     Set
	    \begin{eqnarray*}
	    t_{i j}=\int_{\Omega} u_{j}\left(\left\langle\nabla h, \nabla u_{i}\right\rangle+\frac{u_{i} \Delta h}{2}\right),
	     \end{eqnarray*}
	    then  $$t_{i j}+t_{j i}=0, $$  and
	    \begin{align}\label{rt function}
	    	\int_{\Omega}-2\phi_{i}\left(\left\langle\nabla h, \nabla u_{i}\right\rangle+\frac{u_{i} \Delta h}{2}\right) & =\int_{\Omega}\left(-2 h u_{i}\left\langle\nabla h, \nabla u_{i}\right\rangle-u_{i}^{2} h \Delta h+2 \sum_{j=1}^{k} r_{i j} t_{i j}\right) \nonumber\\
	    	& =\int_{\Omega} u_{i}^{2}|\nabla h|^{2}+2 \sum_{j=1}^{k} r_{i j} t_{i j} .
	    \end{align}
	     Multiplying (\ref{rt function}) by  $\left(\Gamma_{k+1}-\Gamma_{i}\right)^{2} $ and using the Schwarz inequality and (\ref{Rayleigh Rize inequality}) , we get
	     \begin{align}\label{Almost there function}
	     	&	\left(\Gamma_{k+1}-\Gamma_{i}\right)^{2}\left(\int_{\Omega} u_{i}^{2}|\nabla h|^{2}+2 \sum_{j = 1}^{k} r_{i j} t_{i j}\right) \nonumber\\
	     	&	= \left(\Gamma_{k+1}-\Gamma_{i}\right)^{2} \int_{M}-2 \phi_{i}\left(\left\langle\nabla h, \nabla u_{i}\right\rangle+\frac{u_{i} \Delta h}{2}\right) \nonumber\\
	     	&	= \left(\Gamma_{k+1}-\Gamma_{i}\right)^{2} \int_{M}-2 \phi_{i}\left(\left(\left\langle\nabla h, \nabla u_{i}\right\rangle+\frac{u_{i} \Delta h}{2}\right)-\sum_{j = 1}^{k} t_{i j} u_{j}\right) \nonumber\\
	     	&	\leq \delta\left(\Gamma_{k+1}-\Gamma_{i}\right)^{3}|| \phi_{i} \|^{2}+\frac{\left(\Gamma_{k+1}-\Gamma_{i}\right)}{\delta} \int_{M}\left|\left\langle h, \nabla u_{i}\right\rangle+\frac{u_{i} \Delta h}{2}-\sum_{j = 1}^{k} t_{i j} u_{j}\right|^{2} \nonumber\\
	     	&	\left. = \delta\left(\Gamma_{k+1}-\Gamma_{i}\right)^{3} \|\left.\phi_{i}\right|^{2}+\frac{\left(\Gamma_{k+1}-\Gamma_{i}\right)}{\delta} \int_{\Omega}\left(\left\langle\nabla h, \nabla u_{i}\right\rangle+\frac{u_{i} \Delta h}{2}\right)^{2}-\sum_{j = 1}^{k} t_{i j}^{2}\right) \nonumber\\
	     	& \leq \delta\left(\Gamma_{k+1}-\Gamma_{i}\right)^{2}\left(\int_{\Omega}\left(u_{i}^{2}(\Delta h)^{2}+4\left(\left\langle\nabla h, \nabla u_{i}\right\rangle^{2}+u_{i} \Delta h\left\langle\nabla h, \nabla u_{i}\right\rangle\right)-2 u_{i}|\nabla h|^{2} \Delta u_{i}\right)\right. \nonumber\\
	     	&	\left.\quad+\sum_{j = 1}^{k}\left(\Gamma_{i}-\Gamma_{j}\right) r_{i j}^{2}\right)+\frac{\left(\Gamma_{k+1}-\Gamma_{i}\right)}{\delta}\left(\int_{\Omega}\left(\left\langle\nabla h, \nabla u_{i}\right\rangle+\frac{u_{i} \Delta h}{2}\right)^{2}-\sum_{j = 1}^{k} t_{i j}^{2}\right).
	     	\end{align}
	    Summing over  $i$ from 1 to  $k$  for (\ref{Almost there function}) and noticing $$ r_{i j}=r_{j i},\ t_{i j}=-t_{j i},$$ we get
	    \begin{align}
	    	&\sum_{i=1}^{k}\left(\Gamma_{k+1}-\Gamma_{i}\right)^{2} \int_{\Omega} u_{i}^{2}|\nabla h|^{2}-2 \sum_{i, j=1}^{k}\left(\Gamma_{k+1}-\Gamma_{i}\right)\left(\Gamma_{i}-\Gamma_{j}\right) r_{i j} t_{i j} \nonumber\\
	    	&\leq \sum_{i=1}^{k}\left(\Gamma_{k+1}-\Gamma_{i}\right)^{2} \delta \int_{\Omega}\left(u_{i}^{2}(\Delta h)^{2}+4\left(\left\langle\nabla h, \nabla u_{i}\right\rangle^{2}+u_{i} \Delta h\left\langle\nabla h, \nabla u_{i}\right\rangle\right)-2 u_{i}|\nabla h|^{2} \Delta u_{i}\right)\nonumber \\
	    	&\quad+\sum_{i=1}^{k} \frac{\left(\Gamma_{k+1}-\Gamma_{i}\right)}{\delta} \int_{\Omega}\left(\left\langle\nabla h, \nabla u_{i}\right\rangle+\frac{u_{i} \Delta h}{2}\right)^{2} \nonumber\\
	    	&\quad-\sum_{i, j=1}^{k}\left(\Gamma_{k+1}-\Gamma_{i}\right) \delta\left(\Gamma_{i}-\Gamma_{j}\right)^{2} r_{i j}^{2}-\sum_{i, j=1}^{k} \frac{\left(\Gamma_{k+1}-\Gamma_{i}\right)}{\delta} t_{i j}^{2},
	    \end{align}
which implies (\ref{lemma 4.1 function}).
	    \end{proof}
	   	
	 \begin{proof}
	  Now we begin to prove Theorem 1.2 by using the lemma above.\\
	  \indent
	 Let $ X_{\alpha}$, $\alpha=1, \ldots, m$ , be the components of the position vector $X$ of $M$ in $ \mathbf{R}^{m}$ . Taking  $h=X_{\alpha}$  in (\ref{lemma 4.1 function}) and summing over $ \alpha $, we have
	 \begin{align}\label{lemma 4.1's result}
	 	&\sum_{i=1}^{k}\left(\Gamma_{k+1}-\Gamma_{i}\right)^{2} \sum_{\alpha=1}^{m} \int_{\Omega} u_{i}^{2}\left|\nabla X_{\alpha}\right|^{2} \nonumber\\
	 	&\leq \delta \sum_{i=1}^{k+1}\left(\Gamma_{k+1}-\Gamma_{i}\right)^{2} \sum_{\alpha=1}^{m} \int_{\Omega}\left(u_{i}^{2}\left(\Delta X_{\alpha}\right)^{2}+4\left(\left\langle\nabla X_{\alpha}, \nabla u_{i}\right\rangle^{2}\right.\right. \nonumber\\
	 	&\quad\left.\left.+u_{i} \Delta X_{\alpha}\left\langle\nabla X_{\alpha}, \nabla u_{i}\right\rangle\right)-2 u_{i}\left|\nabla X_{\alpha}\right|^{2} \Delta u_{i}\right)\nonumber\\
	 	&\quad+\sum_{i=1}^{k} \frac{\left(\Gamma_{k+1}-\Gamma_{i}\right)}{\delta} \sum_{\alpha=1}^{m} \int_{\Omega}\left(\left\langle\nabla X_{\alpha}, \nabla u_{i}\right\rangle+\frac{u_{i} \Delta X_{\alpha}}{2}\right)^{2},
	 \end{align}
	 Since $M$ is isometrically immersed in $ \mathbf{R}^{m}$, we have
	 \begin{eqnarray*}
	 \sum_{\alpha=1}^{m}\left|\nabla X_{\alpha}\right|^{2}=n,
	 \end{eqnarray*}
	 hence we have
	  \begin{eqnarray*}
	 \sum_{\alpha=1}^{m} \int_{\Omega} u_{i}^{2}\left|\nabla X_{\alpha}\right|^{2}=n.
	\end{eqnarray*}
	 Also, we have
	\begin{align*}
		\Delta\left(X_{1}, \ldots, X_{m}\right) & \equiv\left(\Delta X_{1}, \ldots, \Delta X_{m}\right)=n\mathbf{H} \\
		\sum_{\alpha=1}^{m}\left\langle\nabla X_{\alpha}, \nabla u_{i}\right\rangle^{2} & =\sum_{\alpha=1}^{m}\left(\nabla u_{i}\left(X_{\alpha}\right)\right)^{2}=\left|\nabla u_{i}\right|^{2}
	\end{align*}
	and
	\begin{eqnarray*}
		\sum_{\alpha=1}^{m} \Delta X_{\alpha}\left\langle\nabla X_{\alpha}, \nabla u_{i}\right\rangle=\sum_{\alpha=1}^{m} \Delta X_{\alpha} \nabla u_{i}\left(X_{\alpha}\right)=\left\langle n \mathbf{H}, \nabla u_{i}\right\rangle=0 .
	\end{eqnarray*}
	 Substituting these facts into (\ref{lemma 4.1's result}), we obtain
	 \begin{align}\label{almost there function 2}
	 	n \sum_{i=1}^{k}\left(\Gamma_{k+1}-\Gamma_{i}\right)^{2} \leq & \delta \sum_{i=1}^{k}\left(\Gamma_{k+1}-\Gamma_{i}\right)^{2} \int_{\Omega}\left(n^{2} u_{i}^{2}|\mathbf{H}|^{2}+4\left|\nabla u_{i}\right|^{2}-2 n u_{i} \Delta u_{i}\right) \nonumber\\
	 	& +\sum_{i=1}^{k} \frac{\left(\Gamma_{k+1}-\Gamma_{i}\right)}{\delta} \int_{\Omega}\left(\left|\nabla u_{i}\right|^{2}+\frac{n^{2} u_{i}^{2}|\mathbf{H}|^{2}}{4}\right) \nonumber\\
	 	\leq & \delta \sum_{i=1}^{k}\left(\Gamma_{k+1}-\Gamma_{i}\right)^{2}\left(n^{2} C_{1}\||\mathbf{H}|\|^2_{L^n(\Omega)}+(2 n+4) \right)\Gamma_{i}^{1 / 2} \nonumber\\
	 	& +\sum_{i=1}^{k} \frac{\left(\Gamma_{k+1}-\Gamma_{i}\right)}{\delta}\left(1+\frac{n^{2}C_{1} \||\mathbf{H}|\|^2_{L^n(\Omega)}}{4}\right)\Gamma_{i}^{1 / 2} .
	 \end{align}
	  Here in the last inequality, we have used  the fact that
	   \begin{eqnarray}\label{eigenfunction's ineuqality}
	  \int_{\Omega}\left|\nabla u_{i}\right|^{2}=-\int_{\Omega} u_{i} \Delta u_{i} \leq\left(\int_{\Omega} u_{i}^{2}\right)^{1 / 2}\left(\int_{\Omega}\left(\Delta u_{i}\right)^{2}\right)^{1 / 2}=\Gamma_{i}^{1 / 2} .
	  	\end{eqnarray}
	  Taking
	   \begin{eqnarray*}
	  \delta=\left\{\frac{\sum_{i=1}^{k}\left(\Gamma_{k+1}-\Gamma_{i}\right)\left(\left(1+\frac{n^{2}C_{1} \||\mathbf{H}|\|^2_{L^n(\Omega)}}{4}\right)\Gamma_{i}^{1 / 2} \right)}{\sum_{i=1}^{k}\left(\Gamma_{k+1}-\Gamma_{i}\right)^{2}\left(\left(n^{2}C_{2} \||\mathbf{H}|\|^2_{L^n(\Omega)}+(2 n+4) \right)\Gamma_{i}^{1 / 2}\right)}\right\}^{1 / 2}
	  	\end{eqnarray*}
	  in (\ref{almost there function 2}), one gets (\ref{second result}).
	  \end{proof}
	   \section{Proof of Theorem \ref{Theoerm 1.4}}
	   \begin{proof}
	   From the proof of Theorem 1.1 , we have
        \begin{eqnarray}\label{Ash's function 1}
	   		n \sum_{i=1}^{k}\left(\lambda_{k+1}-\lambda_{i}\right)^{2} \leq \sum_{i=1}^{k}\left(\lambda_{k+1}-\lambda_{i}\right)\left(\int_{\Omega}n^2|\mathbf{H}|^{2}u_{i}^{2}+4 \lambda_{i}\right),
	   	\end{eqnarray}
	   	Using H\"older's inequality, we have
	   	\begin{eqnarray}\label{Holder's inequalty}
	   		\int_{\Omega}|\mathbf{H}|^{2}u_{i}^{2} \leq\left\|u_{i}\right\|_{L^{\frac{2 p}{p-1}}(\Omega)}^{2} \left\|\left | \mathbf{H} \right |  \right \|_{L^{2 p}(\Omega)}^{2}.
	   	\end{eqnarray}
	   	Taking $p=\frac{n}{2} $, using Lemma 2.2, we obtain
	   	\begin{eqnarray}\label{lemma 2.2's result 2}
	   		\left(\int_{\Omega}|u_{i}|^{\frac{2 n}{n-2}} \right)^{\frac{n-2}{n}} \leq C_{s} \int_{\Omega}|\nabla u_{i}|^{2} ,
	   	\end{eqnarray}
	   	where  $C_{s}=\left(\frac{C_{2}}{1-\||H|\|_{L^n(\Omega)} C_{2}} \frac{2(n-1)}{n-2}\right)^{2} $.
	   	Thus
	   	\begin{eqnarray}\label{Simon's inequality 2}
	   		\int_{\Omega}|\mathbf{H}|^{2}u_{i}^{2} &\leq& C_{s}\left\|\nabla u_{i}\right\|_{L^{2}(\Omega)}^{2}\left\|\left | \mathbf{H} \right |  \right \|_{L^{n}\left(\Omega\right)}^{2}\nonumber\\
	   		&=&C_s\||\mathbf{H}|\|^2_{L^n(\Omega)}\lambda_i.
\end{eqnarray}
Substituting (\ref{Simon's inequality 2}) into (\ref{Ash's function 1}), we can prove (\ref{third result}).
 \end{proof}

\section{Proof of Theorem \ref{Theorem 1.5}}
\begin{proof}
	From the proof of Theorem 1.2, we have
	\begin{align}\label{theorem 1.2's result}
		n \sum_{i=1}^{k}\left(\Gamma_{k+1}-\Gamma_{i}\right)^{2} \leq & \delta \sum_{i=1}^{k}\left(\Gamma_{k+1}-\Gamma_{i}\right)^{2} \int_{\Omega}\left(n^{2} u_{i}^{2}|\mathbf{H}|^{2}+4\left|\nabla u_{i}\right|^{2}-2 n u_{i} \Delta u_{i}\right) \nonumber\\
		& +\sum_{i=1}^{k} \frac{\left(\Gamma_{k+1}-\Gamma_{i}\right)}{\delta} \int_{\Omega}\left(\left|\nabla u_{i}\right|^{2}+\frac{n^{2} u_{i}^{2}|\mathbf{H}|^{2}}{4}\right)
	\end{align}
	Using (\ref{eigenfunction's ineuqality}),(\ref{Holder's inequalty}) and (\ref{lemma 2.2's result 2}), we obtain
	\begin{eqnarray}\label{Simon's inequality 3}
		\int_{M}|\mathbf{H}|^{2}u_{i}^{2} &\leq& C_{s}\left\|\nabla u_{i}\right\|_{L^{2}(\Omega)}^{2} \left\|\left | \mathbf{H} \right |  \right \|_{L^{n}\left(\Omega\right)}^{2}\nonumber\\
		&\leq&C_s\left\|\left | \mathbf{H} \right |  \right \|^2_{L^{n}\left(\Omega\right)}\Gamma_i^{\frac12}.
	\end{eqnarray}
	where $C_{s}=\left(\frac{C_{2}}{1-\||\mathbf{H}|\|_{L^n(\Omega)} C_{2}} \frac{2(n-1)}{n-2}\right)^{2} $.
	
Substituting (\ref{Simon's inequality 3}) into (\ref{theorem 1.2's result}) , and taking
	 \begin{eqnarray*}
		\delta=\left\{\frac{\sum_{i=1}^{k}\left(\Gamma_{k+1}-\Gamma_{i}\right)\left(\left(1+\frac{n^{2}C_{s} \left\|\left | \mathbf{H} \right |  \right \|^2_{L^{n}\left(\Omega\right)}}{4}\right)\Gamma_{i}^{1 / 2}\right)}{\sum_{i=1}^{k}\left(\Gamma_{k+1}-\Gamma_{i}\right)^{2}\left(\left(n^{2}C_{s} \left\|\left | \mathbf{H}\right |  \right \|^2_{L^{n}\left(\Omega\right)}+(2 n+4) \right)\Gamma_{i}^{1 / 2}\right)}\right\}^{1 / 2}
	\end{eqnarray*}
	we can prove (\ref{fourth result}).
\end{proof}
\section{Acknowledgement}
 The work is supported by the National NSF of China (Grant no. 12271069). The authors would like to thank the referees for their useful comments and suggestions which make this paper more readable.
\vspace{0.5cm}




	
	{}
	\vspace{1cm}\sc
	
Yong Luo

Mathematical Science Research Center of Mathematics,

Chongqing University of Technology,

Chongqing, 400054, China

{\tt yongluo-math@cqut.edu.cn}

\vspace{1cm}\sc
Xianjing Zheng

Mathematical Science Research Center of Mathematics,

Chongqing University of Technology,

Chongqing, 400054, China

{\tt 1311169535@qq.com}

\end{document}